\documentclass{article}

\usepackage[margin = 1.5in]{geometry}

\usepackage{tikz}
\usepackage{tikz-cd}
\usepackage{url, hyperref}
\usepackage{float}

\tikzset{->-/.style={decoration={
  markings,
  mark=at position .45 with {\arrow{>}}},postaction={decorate}}}

\usepackage{amsmath}
\usepackage{amssymb}
\usepackage{enumitem}
\usepackage{graphicx}
\usepackage{mathdots}
\usepackage{color}
\usepackage{diagbox}
\usepackage{array, makecell}
\usepackage{rotating}
\usepackage{amsthm}

\newtheorem{definition}{Definition}
\newtheorem{theorem}[definition]{Theorem}
\newtheorem{example}[definition]{Example}
\newtheorem{proposition}[definition]{Proposition}
\newtheorem{corollary}[definition]{Corollary}
\newtheorem{lemma}[definition]{Lemma}
\newtheorem{remark}[definition]{Remark}
\newtheorem{notation}[definition]{Notation}
\newtheorem{conjecture}[definition]{Conjecture}

\def\C{\mathbb{C}}
\def\Q{\mathbb{Q}}
\def\Z{\mathbb{Z}}
\def\N{\mathbb{N}}
\def\R{\mathbb{R}}
\def\H{\mathsf{H}}

\def\E{\mathsf{E}}

\def\P{\mathbb{P}}
\def\oM{\overline{\mathcal{M}}}
\def\vT{\mathsf{vTev}^X_{\mathsf{\Gamma}}}
\def\tropT{\mathsf{tropTev}^X_{\mathsf{\Gamma}}}
\def\tropTH{\mathsf{tropTev}^{\mathcal{H}_a}_{\mathsf{\Gamma}}}

\title{Genus $0$ logarithmic and tropical fixed-domain counts for Hirzebruch surfaces}
\author{Alessio Cela and Aitor Iribar L\'opez}
\date{\vspace{-5ex}}

\begin{document}

\maketitle

\begin{abstract}

For a non-singular projective toric variety $X$, the virtual logarithmic Tevelev degrees are defined as the virtual degree of the morphism from the moduli stack of logarithmic stable maps $\oM_{\mathsf{\Gamma}}(X)$ to the product $\oM_{g,n} \times X^n$. In this paper, after proving the genus $0$ correspondence theorem in this setting, we use tropical methods to provide closed formulas for the case in which $X$ is a Hirzebruch surface. In order to do so, we explicitly list all the tropical curves contributing to the count.
\end{abstract}

\tableofcontents

\section{Introduction}


Let $X$ be a non-singular projective variety defined over $\C$ of dimension $r$. Fix integers $g \ge 0$ and $n \geq 1$ such that $2g-3+n>0$, ensuring that the moduli stack $\oM_{g,n}$ of stable curve is well-defined. Fix also an effective curve class $\beta \in H_2(X,\Z)$. Curve counts on $X$ are formulated in Gromov-Witten theory as intersection numbers on the moduli space of stable maps $\oM_{g,n}(X,\beta)$ against 
$$
[\oM_{g,n}(X,\beta)]^{\mathrm{vir}}  \in A_*(\oM_{g,n}(X,\beta))
$$
where $[\oM_{g,n}(X,\beta)]^{\mathrm{vir}}$ is the virtual fundamental class constructed in \cite{BF}. The Tevelev degrees of $X$ are such counts where in additions the domain curve is fixed (and general) and $n$ point insertions are imposed. More precisely, assume the dimensional constraint
$$
\mathrm{vdim}(\oM_{g,n}(X,\beta))= \mathrm{dim}(\oM_{g,n} \times X^n)
$$
holds and let
$$
\tau': \oM_{g,n}(X,\beta) \to \oM_{g,n} \times X^n
$$
be the natural map obtained from the stabilized domain curve and the evaluation morphisms.

\begin{definition}\cite[Definition 1.1]{BP}
    The \textbf{Tevelev degree} $\mathsf{vTev}^X_{g,n,\beta} \in \Q$  of $X$ is defined by the equality
    $$
    \tau'_*[\oM_{g,n}(X,\beta)]^{\mathrm{vir}} = \mathsf{vTev}^X_{g,n,\beta} [\oM_{g,n} \times X^n] \in A_*(\oM_{g,n} \times X^n).
    $$
\end{definition}

Fixed-domain curve counts for Grassmanians have a beautiful story at the intersection between algebraic geometry and physics. They are computed by the celebrated Vafa-Intriligator formula, conjectured by the physicists Vafa and Intriligator \cite{VI} and partially proved by Siebert-Tian \cite{ST} and by Bertram-Daskalopoulos-Wentworth in \cite{B2,B1}, and fully proven by Marian-Oprea in \cite{Mo} using Quot-schemes. The equivalence with the formulation in terms of stable maps was then proven by Marian-Oprea-Pandharipande in \cite{MOP}.
The systematic study of Tevelev degrees for general targets started with \cite{CPS}, motivated by work of Tevelev \cite{Tevelev} on scattering amplitudes in mathematical physics. The paper \cite{CPS} then stimulated a series of subsequent studies \cite{BLLRST,BP,Cela,CL2,CL,FL, Lian, lian_pr,LP}



In this paper, our aim is to extend the notion of Tevelev degrees to the situation where $X$ is a toric variety and any tangency condition with the boundary $\partial X$ is imposed. This is achieved using the moduli stack of logarithmic stable maps \cite{GS}. 

After Mikhalkin's breakthrough \cite{Mik}, a natural correspondence between algebraic curves and tropical curves is expected in certain nice situations. In the recent years, various versions of such correspondence have been proved \cite{NS,Shustin, Tyo, RanganathanI} and many tropical analogs of classical curve counting problems have been proved \cite{GMI,GMII}. Using \cite{Tyo, Goldner}, we show that the correspondence theorem holds in our context when the genus is $0$. Furthermore, we provide simple closed formulas for Hirzebruch surfaces using tropical methods. Our method can be applied to many other geometries. 

\subsection{Logarithmic curve counting with fixed domain}

Assume further that $X$ is a toric variety with fan $\Sigma$. Fix integers $g \ge 0$ and $n,m \geq 1$ and contact order $c$ along the toric buondary $\partial X$ of $X$ (see \cite[Definition 3.1]{GS}). We package the discrete data $(g,n,m,c)$ in the symbol $\mathsf{\Gamma}$, while still assuming the stability condition $2g-2+n>0$.

Let $\oM_\mathsf{\Gamma}(X)$ be the moduli space of genus $g$ and $n+m$ marked logarithmic stable maps $[f : (C,p_1,...,p_n,q_1,...,q_m) \to X]$ having contact order $c$ to the toric boundary divisor along the $m$ marked points $q_1,...,q_m$. This space and its virtual fundamental class $[\oM_{\mathsf{\Gamma}}(X)]^{\mathrm{vir}}$ were constructed in \cite{GS} and shown to be proper in \cite{ACMW}. 

In this paper we deal with logarithmic fixed-domain curve count problems with point insertions at the markings $p_1, \ldots , p_n$.

We set up the discrete data so that the problem has finitely many solutions. In order to make it precise we require some notation. 

\begin{notation}
    Order the components $D_1, \ldots, D_k$ of $\partial X$. Then, we can think of $c$ as the following data:
\begin{enumerate}
    \item[$\bullet$] a function $\varphi: \{1,...,m\} \to \{1, \ldots,k\}$ encoding to which divisor the marking $q_i$ is sent for $i=1, \ldots,m$;
    \item[$\bullet$] $k$ vectors  $\mu_i \in \N_{\geq 0}^{m_i}$ for $i=1,\ldots,k$ defined by 
    $$
    \mu_{i,j}= \text{ multiplicity prescibed by } c \text{ of the $j$-th marked point } q_k \text{ mapping to } D_i.
    $$
    We will denote by $|\mu_i|$ the length of $\mu_i$ for $i=1, \ldots, k$.
\end{enumerate} 
\end{notation}

Assume the dimensional constraint
$$
\mathrm{vdim}(\oM_\mathsf{\Gamma}(X))= \mathrm{dim}( \oM_{g,n} \times X^n)
$$
holds or equivalently that
\begin{equation}\label{dim constraint}
m=r(n+g-1)
\end{equation}
and let 
\begin{equation}\label{tau}      
    \tau: \oM_{\mathsf{\Gamma}}(X) \rightarrow \oM_{g,n} \times X^{n}
\end{equation}
be the canonical morphism obtained from the domain curve 
$
\pi:\oM_{g,n}(X,\beta) \rightarrow \oM_{g,n}$ and the evaluation maps $  \mathrm{ev}:\oM_{\mathsf{\Gamma}}(X) \rightarrow X^n.
$

\begin{definition}
    The \textbf{virtual logarithmic Tevelev degree} $\vT \in \Q$ of $X$ is defined by the equality
    $$
    \tau_*[\oM_{\mathsf{\Gamma}}(X)]^{\mathrm{vir}}= \Bigg( \prod_{i=1}^k \prod_{u \geq 1}|\{ v \ | \ \mu_{i,v}=u \}|!\Bigg) \hspace{0.1cm} \vT  [\oM_{g,n} \times X^n] \in A^0(\oM_{g,n} \times X^n).
    $$
\end{definition}
The factor $\prod_{i=1}^k \prod_{u \geq 0}|\{ v \ | \ \mu_{i,v}=u \}|!$ reflects the possible orderings of the markings $q_j$.

In Theorem \ref{thm: main} below, we compute all the genus $0$ virtual Tevelev degrees for Hirzebruch surfaces using tropical methods.

\subsection{Genus \texorpdfstring{$0$}{} Correspondence theorem}

Suppose that $g=0$ and let $M^{\mathrm{trop}}(\R^r,\mathsf{\Gamma})$ be the moduli space of labelled tropical rational $n$ marked tropical curves  $[h: \mathsf{C} \to \R^r]$ of degree $\Delta$ prescribed by $c$. By definition, $\Delta$ is an ordered list of $m$ vectors $v_i$ in $\R^r$ each parallel to one ray of the fan $\Sigma$ and such that if $c$ prescribes that $q_k$ is the $j$-th marking mapped to $D_i$ then the lattice length of $v_k$ is $\mu_{i,j}$. Our definitions of tropical curves, maps and their moduli spaces is that in \cite[Definition 3.2 and 4.1]{GKM}.

Note that, by Equation \eqref{dim constraint}, we have
$$
| \Delta|=m= r(n-1).
$$

Let
\begin{equation}\label{tautrop}
\mathrm{trop}(\tau): M^{\mathrm{trop}}(\R^r,\mathsf{\Gamma}) \to M_{0,n}^{\mathrm{trop}} \times (\R^r)^n 
\end{equation}
be the canonical morphism obtained from the domain curve and the evaluation map. The map $\mathrm{trop}(\tau)$ is a morphism of equidimensional tropical fans with $M_{0,n}^{\mathrm{trop}} \times (\R^r)^n $ \cite[Definition 2.8]{GKM}. Since $M_{0,n}^{\mathrm{trop}} \times (\R^r)^n$ is irreducible, by \cite[Corollary 2.26]{GKM}, we have a well-defined notion of degree.

\begin{definition}
    Define
    $$
    \tropT = \frac{\mathrm{degree}(\mathrm{trop}(\tau))}{\prod_{i=1}^k \prod_{u \geq 1}|\{ v \ | \ \mu_{i,v}=u \}|!}
    $$
    to be the \textbf{tropical Tevelev degree} of $X$ w.r.t. $\mathsf{\Gamma}$.
\end{definition}

After Mikhalkin's break-through \cite{Mik}, various correspondence theorems have been proved \cite{NS,Shustin, Tyo, RanganathanI}. In our context, we have the following:

\begin{theorem}\label{thm: correspondence}
     Virtual logarithmic tevelev degrees and their corresponding tropical degrees coincide in genus $0$, i.e.
    $$
    \vT=\tropT.
    $$
\end{theorem}
The proof is given in \S\ref{Correspondence} below. 

\subsection{Genus \texorpdfstring{$0$}{} counts for Hirzebruch surfaces}

In the following we specialize to the case when $X= \mathcal{H}_a$ is the Hirzebruch surface $\P( \mathcal{O} \oplus \mathcal{O}(a))$ with $a \geq 1$ and provide closed formulas for the tropical (and so the virtual logarithmic) Tevelev degrees of $\mathcal{H}_a$ with any tangency conditions $c$. 

\begin{notation}\label{notation: rays}
    The fan $\Sigma$ of $\mathcal{H}_a$ has four rays, with associated unit vectors
    $$
    n_1= (-1,a), \
    n_2= (0,1),\
    n_3= (1,0),\text{ and }
    n_4= (0,-1).
    $$
    Denote by $D_1,D_2,D_3$ and $D_4$ the corresponding toric divisors.
\end{notation}


\begin{theorem}\label{thm: main}
    We have 
    \begin{enumerate}
        \item[$\bullet$] if either $|\mu_1|>n-1$ or $|\mu_3|>n-1$, then 
        $$
        \tropTH=0,
        $$
        \item[$\bullet$] otherwise 
        $$
        \tropTH = 
        \Bigg( \prod_{i=1}^4  \frac{|\mu_i|! \prod_{j=1}^{|\mu_i|} \mu_{i,j} }{\prod_{u \geq 1} |\{v \ | \ \mu_{i,v}=u \}|!}\Bigg) a^{n-1-|\mu_2|-|\mu_4|} \binom{n-1-|\mu_4|}{|\mu_2|}
        $$
    \end{enumerate}
\end{theorem}
The proof of this theorem is given in \S\ref{main formula}. 

\begin{remark}\label{rmk: absence of curves}
    The formula above gives zero whenever $ |\mu_2| > n-1-|\mu_4|$. In particular, suppose that
    $\mu_{i,j}=1$ for all $i,j$ and that $a\geq 2$. Then 
    $$
    |\mu_1|=|\mu_3| \text{ and } |\mu_4|=|\mu_2|+(a+1)|\mu_1|
    $$
    so 
    $$
    |\mu_4| +|\mu_2|> \frac{|\Delta|}{2}=n-1
    $$
    and $\mathsf{tropTev}^{\mathcal{H}_a}_{\mathsf{\Gamma}} =0$.
\end{remark} 

A geometric interpretation of this fact is given in \S\ref{sec: absence of curves}.

Suppose $a=1$ and $\mu_2=\emptyset$. Formally, $\Sigma$ reduces to the fan of $\P^2$ and we are counting curves in $\P^2$. Then (the proof of) Theorem \ref{thm: main} also shows the following. 

\begin{theorem}
    We have 
    $$
    \mathsf{tropTev}^{\P^2}_{\mathsf{\Gamma}}=  \prod_{i=1,3,4}  \frac{|\mu_i|! \prod_{j=1}^{|\mu_i|} \mu_{i,j} }{\prod_{u \geq 1} |\{v \ | \ \mu_{i,v}=u \}|!}.
    $$
\end{theorem}

\subsubsection{Description of the curves enumerated in \texorpdfstring{$\mathsf{tropTev}^{\mathcal{H}_a}_{\mathsf{\Gamma}}$}{}}

When $r=2$, we can describe all the curves contributing to $\tropT$. Fix general points $x_1, \ldots, x_n$ in $\R^2$ and fix the stabilized domain curve $\bar{\mathsf{C}}$ in $M_{0,n}^{\mathrm{trop}}$ to have have all lengths equal to $0$. Such a curve is not in the interior of a maximal cone of $M_{0,n}^{\mathrm{trop}}$, but we are allowed to assume so by the intersection theoretic point view presented in \S\ref{sec: intersection theory and tropical defns}.

\begin{proposition}\label{prop: type A and B}
    The curves $[h: \mathsf{C} \to \R^2]$ contributing to $\tropT$ with point insertions $x_1, \ldots ,x_n$ and stabilized domain curve $\bar{\mathsf{C}}$ are all embeddings and the domain curve $\mathsf{C}$ has one of the two shapes in Figure \ref{fig: type A and B}:
    \begin{enumerate}
        \item[A)] in type A there is a central vertex $V$ with the marking $p_1$ attached to it and $n-1$ leaves from $V$ 
        consisting of two bounded edges, one marking and two ends,
        \item[B)] in type B, there still is a central vertex $V$ and $n$ leaves from it of which exactly two consist of one bounded edge, one marking and one unbounded edge and the other $n-2$ leaves are as in type A. 
    \end{enumerate}
    \begin{figure}[h]
    \begin{center}
    \begin{tikzpicture}[xscale=0.22,yscale=0.22]

			

    \draw[dashed, ultra thick, red] (0,0) to (-1.5,1.5);
    \node[red] at (-1.6,1.9) {$p_{1}$};
    \node at (1,0) {$V$};

    \draw [thick] (0,0) to (0,5);
    \draw [thick] (0,5) to (-2,7);
    \draw [thick] (0,5) to (2,7);
    \draw[dashed, thick, red] (1,6) to (0,7);
    \node[red] at (-0.5,7.5) {$p_{i_1}$};

    \draw [thick] (0,0) to (3,3);
    \draw [thick] (3,3) to (3,5);
    \draw [thick] (3,3) to (5,3);
    \draw[dashed, thick, red] (4,3) to (4,4);
    \node[red] at (4,4.5) {$p_{i_2}$};

    \draw [thick] (0,0) to (3,-3);
    \draw [thick] (3,-3) to (5,-3);
    \draw [thick] (3,-3) to (3,-5);
    \draw[dashed, thick, red] (4,-3) to (4,-4);
    \node[red] at (4,-4.5) {$p_{i_3}$};

    \node[thick] at (0.1,-3.9) {\tiny$\bullet$};
    \node at (-1,-3.8) {\tiny$\bullet$};
    \node at (-2,-3.5) {\tiny$\bullet$};

    \node at (0,-7) {Type (A)};

    \node at (19,0) {$V$};
    \draw [thick] (20,0) to (20,5);
    \draw [thick] (20,5) to (18,7);
    \draw [thick] (20,5) to (22,7);
    \draw[dashed, thick, red] (21,6) to (20,7);
    \node[red] at (19.5,7.5) {$p_{i_1}$};

    \draw [thick] (20,0) to (23,3);
    \draw [thick] (23,3) to (23,5);
    \draw [thick] (23,3) to (25,3);
    \draw[dashed, thick, red] (24,3) to (24,4);
    \node[red] at (24,4.5) {$p_{i_2}$};

    \draw[thick] (20,0) to (26,1);
    \draw[dashed, thick, red] (25,0.8) to (24.5,2.3);
    \node[red] at (25.4,2.3) {$p_{i_3}$};

    \draw[thick] (20,0) to (26,-2);
    \draw[dashed, thick, red] (25,-1.5) to (25.5,0);
    \node[red] at (25.8,0) {$p_{i_4}$};

    \draw [thick] (20,0) to (23,-3);
    \draw [thick] (23,-3) to (25,-3);
    \draw [thick] (23,-3) to (23,-5);
    \draw[dashed, thick, red] (24,-3) to (24,-4);
    \node[red] at (24,-4.5) {$p_{i_5}$};

    \node at (20.1,-3.8) {\tiny$\bullet$};
    \node at (19,-3.8) {\tiny$\bullet$};
    \node at (18,-3.5) {\tiny$\bullet$};

    \node at (20,-7) {Type (B)};

    \end{tikzpicture}
    \caption{Shape of the domain curve}\label{fig: type A and B}
    \end{center}
    \end{figure}
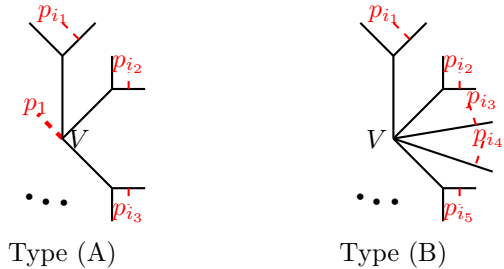    
\end{proposition}
The proof of this proposition is also given in \S\ref{main formula}.
For $X= \mathcal{H}_a$, we will choose points in the following way:
\begin{enumerate}
    \item[$\bullet$] if $|\mu_3|+|\mu_4| \geq n-1$, the point $x_1$ is in the origin $(0,0)$, there are $n-1-|\mu_4|$ points in $\{ x>0, y>0 \}$, $|\mu_3|+|\mu_4|-(n-1)$ in $\{ ax+y>0, y<0\}$ and $n-1-|\mu_3|$ in $\{ x<0, y<0 \}$. Note that if $|\mu_3|>n-1$ or $|\mu_4|>n-1$ then Theorem \ref{thm: main} prescribes $\tropTH=0$.
    \item[$\bullet$] if instead $|\mu_3|+|\mu_4| < n-1$, then again $x_1$ is in the origin $(0,0)$, there are $n-1-|\mu_3|-|\mu_4|$ points in $\{ x<0, ax+y>0 \}$, $|\mu_3|$ in $\{ x>0, y>0 \}$ and $|\mu_4|$ in $\{ x<0, y<0 \}$.
\end{enumerate}
We will then prove that there are
$$
\prod_{i=1}^4 |\mu_i|!\binom{n-1-|\mu_4|}{|\mu_2|}
$$
tropical curves as in Proposition \ref{prop: type A and B} and moreover that each of such curves contributes with multiplicity
$$
a^{n-1-|\mu_2|-|\mu_4|}\prod_{j=1}^{|\mu_i|} \mu_{i,j}.
$$ 
to $\tropTH$.

\begin{example}\label{example 1}
    Suppose $a=2$ and $\mu_1=(1,2)$, $\mu_2=(3)$, $\mu_3=(1,1,1)$ and $\mu_4=(4,4)$. So in this case $|\mu_3|+|\mu_4| \geq n-1$. We list in Figure \ref{fig: example 1} the  $4$ contributing curves, all of which are of type A.
    \begin{figure}[h]
    \begin{center}
    \begin{tikzpicture}[xscale=0.15,yscale=0.15]

			
	   \draw [dashed , black] (-5,5) to        (15,5);

        \draw [dashed, black] (5,-5) to (5,15);

        \draw [dashed, black] (5,5) to (0,15);



        \draw [ultra thick] (5,5) to (5,8);
        \draw [ultra thick] (5,8) to (15,8);
        \draw [ultra thick] (5,8) to (1.5,15);
        \node at (7,16) {$3$};
        \node at (16,8) {$1$};

        \draw [ultra thick] (5,5) to (7,6);
        \draw [ultra thick] (7,6) to (7,15);
        \draw [ultra thick] (7,6) to (15,6);
        \node at (16,6) {$1$};
        \node at (1.6,16) {$1$};

        \draw [ultra thick] (5,5) to (7,-3);
        \draw [ultra thick] (7,-3) to (7,-5);
        \draw [ultra thick] (7,-3) to (15,-3);
        \node at (16,-3) {$1$};
        \node at (7,-6) {$4$};


        \draw [ultra thick] (5,5) to (3,5);
        \draw [ultra thick] (3,5) to (3,-5);
        \draw [ultra thick] (3,5) to (-3,15);
        \node at (3,-6) {$4$};
        \node at (-4,16) {$2$};

        \node[red] at (5,5) {$\bullet$};
        \node[red] at (6.5,4.2) {$x_6$};

        \node[red] at (7,11) {$\bullet$};
        \node[red] at (8.5,11) {$x_1$};

        \node[red] at (13,8) {$\bullet$};
        \node[red] at (14,9) {$x_2$};

        \node[red] at (13,-3) {$\bullet$};
        \node[red] at (14,-2) {$x_3$};

        \node[red] at (3,-4) {$\bullet$};
        \node[red] at (4,-3) {$x_4$};


	   \draw [dashed , black] (20,5) to        (40,5);

        \draw [dashed, black] (30,-5) to (30,15);

        \draw [dashed, black] (30,5) to (25,15);



        \draw [ultra thick] (30,5) to (30,11);
        \draw [ultra thick] (30,11) to (40,11);
        \draw [ultra thick] (30,11) to (28,15);
        \node at (36,16) {$3$};
        \node at (41,8) {$1$};

        \draw [ultra thick] (30,5) to (36,8);
        \draw [ultra thick] (36,8) to (36,15);
        \draw [ultra thick] (36,8) to (40,8);
        \node at (41,11) {$1$};
        \node at (28,16) {$1$};


        \draw [ultra thick] (30,5) to (32,-3);
        \draw [ultra thick] (32,-3) to (32,-5);
        \draw [ultra thick] (32,-3) to (40,-3);
        \node at (41,-3) {$1$};
        \node at (32,-6) {$4$};


        \draw [ultra thick] (30,5) to (28,5);
        \draw [ultra thick] (28,5) to (28,-5);
        \draw [ultra thick] (28,5) to (22,15);
        \node at (28,-6) {$4$};
        \node at (21,16) {$2$};


        \node[red] at (30,5) {$\bullet$};
        \node[red] at (31.5,4.2) {$x_6$};

        \node[red] at (32,11) {$\bullet$};
        \node[red] at (33,11.8) {$x_1$};

        \node[red] at (38,8) {$\bullet$};
        \node[red] at (39,9) {$x_2$};

        \node[red] at (38,-3) {$\bullet$};
        \node[red] at (39,-2) {$x_3$};

        \node[red] at (28,-4) {$\bullet$};
        \node[red] at (29,-3) {$x_4$};

			
	   \draw [dashed , black] (-5,-20) to  (15,-20);

        \draw [dashed, black] (5,-30) to (5,-10);

        \draw [dashed, black] (5,-20) to (0,-10);



        \draw [ultra thick] (5,-20) to (4.24,-17);
        \draw [ultra thick] (4.25,-17) to (15,-17);
        \draw [ultra thick] (4.25,-17) to (0.75,-10);
        \node at (7,-9) {$3$};
        \node at (16,-17) {$1$};

        \draw [ultra thick] (5,-20) to (7,-19);
        \draw [ultra thick] (7,-19) to (7,-10);
        \draw [ultra thick] (7,-19) to (15,-19);
        \node at (16,-19) {$1$};
        \node at (1,-9) {$2$};

        \draw [ultra thick] (5,-20) to (7,-28);
        \draw [ultra thick] (7,-28) to (7,-30);
        \draw [ultra thick] (7,-28) to (15,-28);
        \node at (16,-28) {$1$};
        \node at (7,-31) {$4$};


        \draw [ultra thick] (5,-20) to (3,-24);
        \draw [ultra thick] (3,-24) to (3,-30);
        \draw [ultra thick] (3,-24) to (-4,-10);
        \node at (3,-31) {$4$};
        \node at (-4,-9) {$1$};

        \node[red] at (5,-20) {$\bullet$};
        \node[red] at (6.5,-20.8) {$x_6$};

        \node[red] at (7,-14) {$\bullet$};
        \node[red] at (8.5,-14) {$x_1$};

        \node[red] at (13,-17) {$\bullet$};
        \node[red] at (14,-16) {$x_2$};

        \node[red] at (13,-28) {$\bullet$};
        \node[red] at (14,-27) {$x_3$};

        \node[red] at (3,-29) {$\bullet$};
        \node[red] at (4,-28) {$x_4$};

			
	   \draw [dashed , black] (20,-20) to  (40,-20);

        \draw [dashed, black] (30,-30) to (30,-10);

        \draw [dashed, black] (30,-20) to (25,-10);



        \draw [ultra thick] (30,-20) to (28.5,-14);
        \draw [ultra thick] (28.5,-14) to (40,-14);
        \draw [ultra thick] (28.5,-14) to (26.8,-10);
        \node at (36,-9) {$3$};
        \node at (41,-17) {$1$};

        \draw [ultra thick] (30,-20) to (36,-17);
        \draw [ultra thick] (36,-17) to (36,-10);
        \draw [ultra thick] (36,-17) to (40,-17);
        \node at (41,-14) {$1$};
        \node at (27,-9) {$2$};
        

         \draw [ultra thick] (30,-20) to (32,-28);
        \draw [ultra thick] (32,-28) to (32,-30);
        \draw [ultra thick] (32,-28) to (40,-28);
        \node at (41,-28) {$1$};
        \node at (32,-31) {$4$};


        \draw [ultra thick] (30,-20) to (28,-24);
        \draw [ultra thick] (28,-24) to (28,-30);
        \draw [ultra thick] (28,-24) to (21,-10);
        \node at (28,-31) {$4$};
        \node at (21,-9) {$1$};

        \node[red] at (30,-20) {$\bullet$};
        \node[red] at (31.5,-20.8) {$x_6$};

        \node[red] at (32,-14) {$\bullet$};
        \node[red] at (33.5,-13) {$x_1$};

        \node[red] at (38,-17) {$\bullet$};
        \node[red] at (39,-16) {$x_2$};

        \node[red] at (38,-28) {$\bullet$};
        \node[red] at (39,-27) {$x_3$};

        \node[red] at (28,-29) {$\bullet$};
        \node[red] at (29,-28) {$x_4$};
        
    \end{tikzpicture}
    \caption{The $4$ tropical curves contributing to $\mathsf{tropTev}^{\mathcal{H}_a}_{\mathsf{\Gamma}}$ in Example \ref{example 1}}\label{fig: example 1}
    \end{center}
    \end{figure}
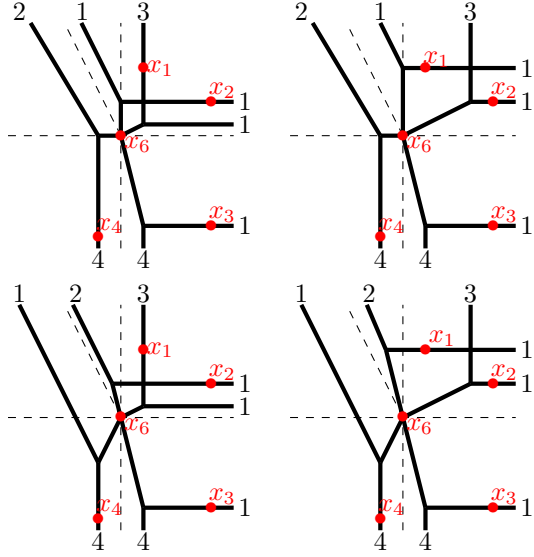
\end{example}

\begin{example}\label{example 2}
    Suppose $a=1$ and $\mu_1=(1,1,1)$, $\mu_2=(1)$, $\mu_3=(3)$ and $\mu_4=(4)$. In this case $|\mu_3|+|\mu_4| \geq n-1$. We list in Figure \ref{fig: example 2} below the $2$ contributing curves: one of type A and one of time B.
    \begin{figure}[h]
    \begin{center}
    \begin{tikzpicture}[xscale=0.15,yscale=0.15]

	\draw [dashed , black] (-5,0) to (15,0);
    \draw [dashed, black] (5,10) to (5,-10);
    \draw [dashed, black] (5,0) to (-5,10);
    
    \draw[ultra thick] (5,0) to (3,4);
    \draw[ultra thick] (3,4) to (-1,8);
    \node at (-2,9) {$1$};
    \draw[ultra thick] (3,4) to (3,10);
    \node[red] at (3,10) {$\bullet$};
    \node[red] at (4.5,10) {$x_2$};
    \node at (3,12) {$1$};


    \draw[ultra thick] (5,0) to (9,2);
    \draw[ultra thick] (9,2) to (1,10);
    \node at (0,11) {$1$};
    \draw[ultra thick] (9,2) to (15,2);
    \node[red ] at (14,2) {$\bullet$};
    \node[red] at (14,3.5) {$x_3$};
    \node at (16,2.6) {$3$};


    \draw[ultra thick] (5,0) to (4,-4);
    \draw[ultra thick] (4,-4) to (4, -10);
    \draw[ultra thick] (4,-4) to (-5, 5);
    \node at (-6,6) {$1$};
    \node[red] at (4,-9) {$\bullet$};
    \node[red] at (5.8,-9) {$x_4$};
    \node at (4,-11) {$4$};

    \node[red] at (5,0) {$\bullet$};
    \node[red] at (7.2,0) {$x_1$};

	\draw [dashed , black] (20,0) to (40,0);
    \draw [dashed, black] (30,10) to (30,-10);
    \draw [dashed, black] (30,0) to (20,10);


    \draw[ultra thick] (28,10) to (45,-7);
    \draw[ultra thick] (45,-7) to (39,-9);
    \draw[ultra thick] (45,-7) to (50,-7);
    \node at (51,-7) {$3$};
    \node[red] at (28,10) {$\bullet$};
    \node[red] at (29.5,10) {$x_2$};
    \node at (27.5,11.5) {$1$};


    \draw[ultra thick] (20,10) to (39,-9);
    \node at (19,11) {$1$};
    \draw[ultra thick] (39,-9) to (39,10);
    \node at (39,11) {$1$};

    
    \node[red ] at (39,2) {$\bullet$};
    \node[red] at (40.5,3.5) {$x_3$};


    \draw[ultra thick] (39,-9) to (37.4,-17);
    \draw[ultra thick] (37.4,-17) to (25,-4.2);
    \node at (24,-3.2) {$1$};
    \draw[ultra thick] (37.4,-17) to (37.4,-18);
    \node at (37.4,-19) {$4$};

    \node[red] at (29,-9) {$\bullet$};
    \node[red] at (31,-9) {$x_4$};
    \node[red] at (30,0) {$\bullet$};
    \node[red] at (31.5,1) {$x_1$};
    
    \end{tikzpicture}
    \caption{The $2$ tropical curves contributing to $\tropTH$ in Example \ref{example 2}}
        \label{fig: example 2}
    \end{center}
    \end{figure}
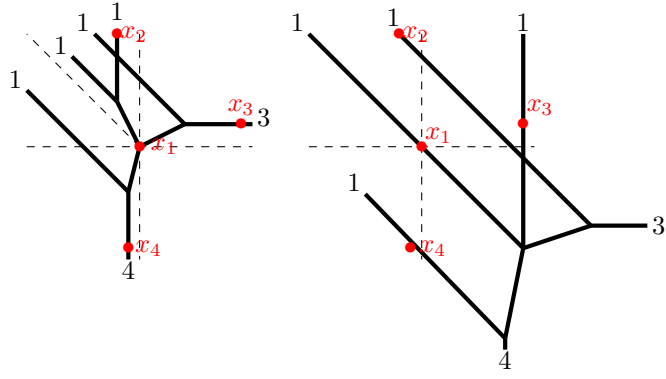
\end{example}

\subsection{Comparison of virtual fundamental classes for maps to Hirzebruch surfaces}

We can use fixed-domain curve counts to distinguish the virtual fundamental class of the moduli spaces of logarithmic and stable maps to Hirzebruch surfaces.

More precisely, let $X=\mathcal{H}_a$ be a Hirzebruch surface and $\beta \in H_2(X,\Z)$ be an effective curve class. Let $c$ be defined by 
$$
\mu_i=\underbrace{( 1, \ldots ,1)}_{\beta \cdot D_i \text{ times} }
$$
for $i=1,2,3,4$ and let $n \in \N$ be such that the dimensional constraint \eqref{dim constraint} holds. Consider the natural (proper) morphism
$$
\alpha: \overline{\mathcal{M}}_\mathsf{\Gamma} (\mathcal{H}_a) \to \oM_{0,n}(\mathcal{H}_a,\beta)
$$ 
of virtually equidimensional Deligne-Mumford stacks. 

\begin{theorem}\label{thm: application}
    In the following cases:
    \begin{enumerate}
        \item $a=2j$ where $j \in \Z_{\geq 1}$ and $\beta=d[(j+1)D_1+D_2]$ for $d>0$ such that $2d=n-1$; or 
        \item $a=2j+1$ where $j \in \Z_{\geq 1}$ and $\beta=[j(d-k)+d] D_1 +(d-k) D_2$ for $d$ and $k$ integers such that $0 \leq k \leq d$, $0 \leq k \leq n-1-d$ and $3d-k=2(n-1)$;
    \end{enumerate} 
    we have 
    $$
    \alpha_*[\oM_\mathsf{\Gamma}(\mathcal{H}_a)]^{\mathrm{vir}} \neq [\oM_{0,n}(\mathcal{H}_a,\beta)]^{\mathrm{vir}}.
    $$
\end{theorem}

This is achieved in \S\ref{proof of applications} by comparing the corresponding Tevelev degrees and using the results of \cite{CL}.

As already observed in \cite{SPIELBERG}, Hirzebruch surfaces provide an excellent example for the fact that in general Gromov–Witten invariants might
well count curves in the boundary components of the moduli spaces. The proof of theorem \ref{thm: main} and Theorem \ref{thm: application} show that logarithmic stable maps behave better from this point of view.

\subsection{Further directions}

Our approach for computing $\tropT$ for Hirzebruch surfaces should generalize to other geometries and higher dimensional varieties.

Higher dimensional generalizations of $\mathcal{H}_a$ includes $\P^1$-bundles $\P(\mathcal{O}_{\P^{r}} \oplus \mathcal{O}_{\P^{r}}(a))$ over $\P^{r}$, for which we conjecture the following formula to hold.

\begin{conjecture}
    Let $X= \P (\mathcal{O}_{\P^r}\oplus \mathcal{O}_{\P^r}(a))$ and let $D_1, \ldots , D_{r+1}$ be the fibers over the invariant hyperplanes $\{x_1=0\}, \ldots , \{x_{r+1}=0\}$, and let $D_{r+2}, D_{r+3}$ be the zero section and the infinity section, respectively. Then, when $\vT$ is not $0$,
    $$
    \vT= \tropT = \left(\prod_{i=1}^{r+3}\frac{|\mu_i|!\prod _{j=0}^{|\mu_{ij}|}\mu_{i,j}}{\prod_{u \geq 1} |\{v \ | \ \mu_{i,v}=u \}|!}\right) a^{(n-1)- |\mu_{r+2}|- |\mu_{r+3}|} {n-1-|\mu_{r+2}|\choose |\mu_{r+3}|}.
    $$
\end{conjecture}



The Hirzebruch surface $\mathcal{H}_1$ is isomorphic to the blow-up of $\P^2$ at one point. In \cite{CL2}, the authors computed the geometric degrees with simple incidence conditions with the toric boundary for blowups of $\P^r$ at up to $r+1$ points. We also conjecture the following generalization of that formula ho hold.

\begin{conjecture}
    Let $X $ be the blowup of $\P^r$ at $r$ of the torus fixed points, and let $D_1, \ldots , D_r$ be the exceptional divisors of $[0:1:\ldots :0], \ldots , [0: \ldots :1]$, and $D_{r+1}, \ldots, D_{2r+1}$ the strict transforms of the linear subspaces $\{x_1 =0\}, \ldots , \{x_{r+1}=0\}$ of $\P^r$. Then if $\vT$ is nonzero,
    $$
    \vT= \tropT = \left(\prod_{i=0}^{2r+1}\frac{|\mu_i|!\prod _{j=0}^{|\mu_{ij}|}\mu_{i,j}}{\prod_{u \geq 1} |\{v \ | \ \mu_{i,v}=u \}|!}\right)\prod_{i=1}^r{n-1-|\mu_{i+r+1}|\choose |\mu_i|}.
    $$
\end{conjecture}

Finally, we expect a more complicated formula could be obtained with our method for the blow-up of $\P^r$ at the $r+1$ torus fixed points (with any tangencies with the toric boundary).

\subsection*{Acknowledgments}
This project began with the participation of the first author in the MSRI summer school titled "Tropical Geometry" at St. Mary's College in Moraga, California, in August 2022. The first author is deeply grateful to the organizers, Renzo Cavalieri, Hannah Markwig, and Dhruv Ranganathan, for teaching him tropical and logarithmic geometry. We would also like to thank these three researchers for their invaluable assistance with this project during their visit to ETH Zurich in the spring semester of 2023. Lastly, we thank Gavril Farkas, Carl Lian, Rahul Pandharipande, and Johannes Schmitt for several useful discussions regarding fixed domain curve counts. A.C. received support from SNF-200020-182181. A.I.L. was supported by
ERC-2017-AdG-786580-MACI. The project received funding from the
European Research Council (ERC) under the European Union Horizon
2020 research and innovation programme (grant agreement 786580).

\section{The correspondence theorem}\label{Correspondence}

In this section, we assume familiarity with the intersection theory on balanced fans (see \cite{Goldner} for an introduction). The starting point to prove the correspondence theorem \ref{thm: correspondence} are \cite{Tyo, Goldner}.

\begin{lemma}\label{lemma: reduction to crossratios}
    The natural maps 
    $
     \oM_{0,n} \to \prod_{i=4}^n \oM_{0,\{1,2,3,i\}}$
    and
    $M^{\mathrm{trop}}_{0,n} \to \prod_{i=4}^n M_{0,\{1,2,3,i\}}^{\mathrm{trop}}
    $
    have degree $1$.
\end{lemma}
\begin{proof}
    The statement is clear for the first map and for the second follows from the first map having degree $1$ and \cite[Theorem 4.1]{AndreasGross}.
\end{proof}

The degree of the map \ref{tau} is then equal to the degree of the map
\begin{equation}
 \bigg( \prod_{i=4}^n \mathrm{ft}_i \bigg) \times \mathrm{ev}: \oM_\mathsf{\Gamma}(X) \to \prod_{i=4}^n \oM_{0,\{1,2,3,i\}} \times X^n
\end{equation}
obtained by the forgetful morphisms $\mathrm{ft}_i : \oM_\mathsf{\Gamma}(X) \to \oM_{0,\{1,2,3,i\}}$ for $i=4, \ldots, n$ and the evaluation map. 

Similarly, the degree of the map \ref{tautrop} is equal to the degree of 
\begin{equation}\label{ft ev map}
 \bigg( \prod_{i=4}^n \mathrm{ft}_i \bigg) \times \mathrm{ev}: M^{\mathrm{trop}}(\R^r,\mathsf{\Gamma}) \to \prod_{i=4}^n M_{0,\{1,2,3,i\}}^{\mathrm{trop}} \times (\R^r)^n
\end{equation}

The correspondence theorem \ref{thm: correspondence} will now follow from \cite[Theorem 5.1]{Tyo}, a generalization of \cite[Proposition 3.5]{Goldner} to higher dimensions and next two lemmas.


\begin{lemma}\label{lemma: irreducibility of moduli spaces}
    In genus $0$ the moduli spaces $\oM_\mathsf{\Gamma}(X)$ are irreducible, generically of expected dimension. Moreover,
    $$
    [\oM_\mathsf{\Gamma}(X)]= [\oM_\mathsf{\Gamma}(X)]^{\mathrm{vir}} \in A_*(\oM_\mathsf{\Gamma}(X)).
    $$
\end{lemma}

\begin{proof}
    Let $\mathfrak{M}_{0,n}$ be the moduli space of genus $0$ and $n$-marked prestable curves, endowed with the logarithmic structure given by $\partial \mathfrak{M}_{0,n+m}$. Let also $\mathcal{L}og_{\mathfrak{M}_{0,n+m}}$ be the stack constructed in \cite{OLSSON}. It has pure dimension $-3+n+m$. Consider the natural (strict) logarithmic map
    $$
    \varphi: \oM_{\mathsf{\Gamma}}(X) \to \mathcal{L}og_{\mathfrak{M}_{0,n+m}}
    $$
    The relative perfect obstruction theory (see \cite[Section 5]{GS}) is
    $$
    E^\bullet=Rp_*(f^*T_X^{\mathrm{log}})^ \vee \to L^\bullet_{\oM_{\mathsf{\Gamma}}(X) /\mathcal{L}og_{\mathfrak{M}_{0,n+m}}}
    $$ where $p:\mathcal{C} \to \oM_{\mathsf{\Gamma}}(X)$ is the universal curve and $f: \mathcal{C} \to X$ is the universal map. In our situation, we have $T_X^{\mathrm{log}}= \mathcal{O}_X^r$ and therefore $E^\bullet \cong \Omega_{\oM_{\mathsf{\Gamma}}(X) /\mathcal{L}og_{\mathfrak{M}_{0,n+m}}}$ is a vector bundle with fiber over $[f:C \to X]$ given by $(H^0(C, \mathcal{O})^r)^\vee$. It follows that $\varphi$ is smooth (and therefore log-smooth) and that $[\oM_\mathsf{\Gamma}(X)]= [\oM_\mathsf{\Gamma}(X)]^{\mathrm{vir}}$. The irreducibility statement is \cite[Proposition 3.3.5]{RanganathanI}.
\end{proof}


    

\begin{remark}
    The proof of \cite[Proposition 3.5]{Goldner} straightforwardly generalizes to the higher-dimensional setting. In the proof of Theorem \ref{thm: correspondence} below, we will use its higher-dimensional generalization.
\end{remark}

\begin{proof}[Proof of Theorem \ref{thm: correspondence}]
    By Lemma \ref{lemma: irreducibility of moduli spaces}, $\vT$ equals the \textit{geometric} numbers of logarithmic stable maps from a fixed general curve $(C, p_1,...,p_n)$ to $X$ with contact orders prescribed by $c$. By Lemma \ref{lemma: reduction to crossratios}, \cite[Theorem 5.1]{Tyo} and the higher-dimensional generalization of \cite[Proposition 3.5]{Goldner}, the number of such maps where in addition the domain curve is smooth and the image of the map does not meet any toric point of $X$ is equal to the degree of $\mathrm{trop}(\tau)$. In order to conclude it is then enough to exclude contributions in $\vT$ from $\partial \oM_\mathsf{\Gamma}(X)$ and from the locus $B$ of maps containing toric points in their image. The boundary $\partial \oM_\mathsf{\Gamma}(X)$ and $B$ are both proper subsets of $\oM_\mathsf{\Gamma}(X)$ \cite[Proposition 3.3.3]{RanganathanI} so, by lemma \ref{lemma: irreducibility of moduli spaces}, they cannot dominate $\oM_{0,n} \times X^n$ for dimensional reasons.   
\end{proof}

\section{Genus \texorpdfstring{$0$} {} tropical count for Hirzebruch surfaces}\label{main formula}

In this section we calculate $\tropT$ for tropical maps to Hirzebruch surfaces.

\subsection{Tropical curves and intersection theory}\label{sec: intersection theory and tropical defns}

The moduli spaces $M^{\mathrm{trop}}_{0,n}$ and $M^{\mathrm{trop}}(\R^n, \Gamma)$ are fans and therefore the machinery of tropical intersection theory developed in \cite{Allermann}, can be applied to it.\\
All the results in this sections are valid for any toric surface, or in general a toric variety of any dimension after appropriate modifications, specially to Corollary \ref{description curves}.\\ 

\begin{notation}
   For a morphism of fans, there is a notion of pullback of Cartier divisors. For a point $\bar{x} =(\bar{x}_1,\bar{x}_2) \in \R^2$ and any affine cycle $Z$ of $M^{\mathrm{trop}}(\R^n, \Gamma)$, we make the abbreviation 
   $$
   \mathrm{ev}^*(p)\cdot Z = (\mathrm{pr}_1 \circ\mathrm{ev})^*(\bar{x}_1)\cdot (\mathrm{pr}_2 \circ\mathrm{ev})^*(\bar{x}_2)\cdot Z,
   $$
   where $\mathrm{pr}_i: \R^2 \to \R$ are the two projections and $\bar{x}_i \in \R$ is regarded as a Cartier divisor for $i=1,2$.
\end{notation}

\begin{lemma}
    The degree of $\operatorname{trop} (\tau)$ is equal to the degree of the tropical $0$-cycle 
    \begin{equation}\label{eqn: our count}
        \prod_{i=4}^n \mathrm{ft}_i^*(0) \cdot \prod_{i=1}^n \mathrm{ev}_i^*(x_i) \cdot M^{\mathrm{trop}}(\R^2, \Gamma),
    \end{equation}
    where $\mathrm{ft}_i : M^{\mathrm{trop}}(\R^2, \Gamma) \to M^{\mathrm{trop}}_{0,\{1,2,3,i\}}$, where $0$ represents the unique $4$-valent curve in $M_{0,4}^{\mathrm{trop}}$, and $x_i \in \R^2$.
\end{lemma}
\begin{proof}
    Let $ ((\lambda_1, \ldots, \lambda_{n-3}),(x_1, \ldots, x_n)) \in \left(M^{\mathrm{trop}}_{0,4}\right)^{n-3} \times (\R^{2})^n$ be a general point and let $g=\bigg( \prod_{i=4}^n \mathrm{ft}_i \bigg) \times \mathrm{ev}$ be the morphism of tropical fans in \ref{ft ev map}. The equality
    $$
    \deg(g) = \deg \left( \prod_{i=4}^n \mathrm{ft}_i^*(\lambda_i) \cdot \prod_{i=1}^n \mathrm{ev}_i^*(x_i) \cdot M^{\mathrm{trop}}(\R^2, \Gamma)\right)
    $$
    follows essentially from \cite[Lemma 1.2.9]{Rau}, and the equality
    $$
    \deg \left( \prod_{i=4}^n \mathrm{ft}_i^*(\lambda_i) \cdot \prod_{i=1}^n \mathrm{ev}_i^*(x_i) \cdot M^{\mathrm{trop}}(\R^2, \Gamma)\right)=\deg \left( \prod_{i=4}^n \mathrm{ft}_i^*(0) \cdot \prod_{i=1}^n \mathrm{ev}_i^*(x_i) \cdot M^{\mathrm{trop}}(\R^2, \Gamma)\right)
    $$
    follows from the fact that pullbacks, intersections and the degree are defined up to rational equivalence, and all points in $M^{\mathrm{trop}}_{0,4}$ are rationally equivalent.
\end{proof}


\begin{lemma}\label{lemma: central vertex}
    Consider the tropical affine cycle
    $$
    Z = \prod_{i=4}^n \mathrm{ft}_i^*(0)\cdot M^{\mathrm{trop}}(\R^2, \Gamma)
    $$
    Then the combinatorial type of any top-dimensional polyhedron of $Z$ is such that one vertex $V$ is $n$-valent, the rest of them are 3-valent, and for each each edge is contained in the unique path to a unique marking. Moreover, the weight of all these polyhedra is $1$.
\end{lemma}

\begin{proof}
    Let $[h:\mathsf{C} \to \R^2]$ be such a curve. Since for any $i=4, \ldots, n$ we have $\mathrm{ft}_i([h]) =0$, for each $i$ there must be a unique vertex $V_i$ in $\mathsf{C}$ and $4$ different edges $e_1, e_2, e_3, e_i$ attached to $V_i$ that are part of the unique path between $V_i$ and $p_1$, $p_2$, $p_3$ and $p_i$, respectively. In particular, $V_i$ must be contained in the unique path joining $p_1$, $p_2$, and in the unique path containing $p_1$, $p_3$. Therefore, there exists a vertex $V$ such that $V_i = V$ for all $i$. By \cite[Lemma 3.11]{Goldner}, for any vertex $W$ of $\mathsf{C}$
    $$
    \mathrm{val}(W) = 3+ |\{i \mid V_i = W\}|,
    $$
    And this implies the first assertion of the lemma.\\
    The multiplicities of such polyhedra are also computed in the same lemma \cite[Lemma 3.11]{Goldner}, and, in the case we are concerned with, they are all $1$. This amounts to say that there is only one (see Figure \ref{fig: resolving cross ratios} below) way of resolving the cross ratios $(p_1, p_2;p_3, p_i)$ for $i = 4, \ldots, n$.
\end{proof}

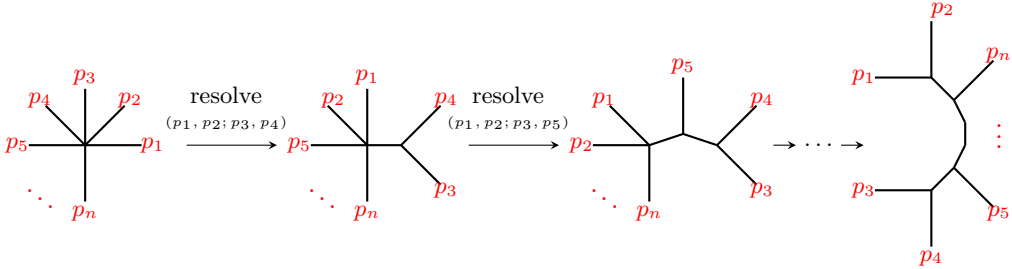
\begin{figure}[h]
    \centering
    \begin{tikzpicture}[xscale=0.15,yscale=0.15]
        \draw [thick] (0,0) to (5,0);
        \node[red] at (6,0) {\small $p_1$};
        \draw [thick] (0,0) to (3.5,3.5);
        \node[red] at (4,4) {\small $p_2$};
        \draw [thick] (0,0) to (0,5);
        \node[red] at (0,6) {\small $p_3$};
        \draw [thick] (0,0) to (-3.5,3.5);
        \node[red] at (-4,4) {\small $p_4$};
        \draw [thick] (0,0) to (-5,0);
        \node[red] at (-6,0) {\small $p_5$};
        \node[red] at (-4,-4) {$\ddots$};
        \draw [thick] (0,0) to (0,-5);
        \node[red] at (0,-6) {\small $p_n$};

        \draw [-stealth](9,0) -- (17,0);
        \node at (12.5, 4.5){\small resolve};
        \node at (12.5, 2){\tiny $(p_1, p_2; p_3, p_4)$};

        \draw [thick] (25,0) to (25,5);
        \node[red] at (25,6) {\small $p_1$};
        \draw [thick] (25,0) to (21.5,3.5);
        \node[red] at (22,4) {\small $p_2$};
        \draw [thick] (25,0) to (20,0);
        \node[red] at (19,0) {\small $p_5$};
        \node[red] at (21,-4) {$\ddots$};
        \draw [thick] (25,0) to (25,-5);
        \node[red] at (25,-6) {\small $p_n$};
        \draw[thick] (25,0) to (28,0);
        \draw[thick] (28,0) to (31.5,3.5);
        \node[red] at (32,4) {\small $p_4$};
        \draw[thick] (28,0) to (31.5,-3.5);
        \node[red] at (32,-4) {\small $p_3$};

        \draw [-stealth](34,0) -- (42,0);
        \node at (37.5, 4.5){\small resolve};
        \node at (37.5, 2){\tiny $(p_1, p_2; p_3, p_5)$};

        \draw [thick] (50,0) to (46.5,3.5);
        \node[red] at (46,4) {\small $p_1$};
        \draw [thick] (50,0) to (45,0);
        \node[red] at (44,0) {\small $p_2$};
        \node[red] at (46,-4) {$\ddots$};
        \draw [thick] (50,0) to (50,-5);
        \node[red] at (50,-6) {\small $p_n$};
        \draw[thick] (50,0) to (53,1);
        \draw[thick] (53,1) to (53,6);
        \node[red] at (53,7) {\small $p_5$};
        \draw[thick] (53,1) to (56, 0);
        \draw[thick] (56,0) to (59.5, 3.5);
        \node[red] at (60,4) {\small $p_4$};
        \draw[thick] (56,0) to (59.5, -3.5);
        \node[red] at (60,-4) {\small $p_3$};

        \draw [-stealth](61,0) -- (63,0);
        \node at (65,0){\small $\ldots$};
        \draw [-stealth](67,0) -- (69,0);

        \draw[thick] (70, 6) to (75, 6);
        \node[red] at (69, 6){\small $p_1$};
        \draw[thick] (75, 11) to (75, 6);
        \node[red] at (76, 12){\small $p_2$};
        \draw[thick] (75, 6) to (77, 4);
        \draw[thick] (80.5, 7.5) to (77, 4);
        \node[red] at (81, 8){\small $p_n$};
        \draw[thick] (77, 4) to (78, 2);
        \draw[thick] (78, 2) to (78, 0);
        \node[red] at (81, 1.5) {$\vdots$};
        \draw[thick] (78, 0) to (77, -2);
        \draw[thick] (77, -2) to (80.5, -5.5);
        \node[red] at (81, -6){\small $p_5$};
        \draw[thick] (77, -2) to (75, -4);
        \draw[thick] (75, -4) to (75, -9);
        \node[red] at (75, -10){\small $p_4$};
        \draw[thick] (70, -4) to (75, -4);
        \node[red] at (69, -4){\small $p_3$};
    \end{tikzpicture}
    \caption{Resolving the cross-ratios in the proof of Lemma \ref{lemma: central vertex}.}
    \label{fig: resolving cross ratios}
\end{figure}

\begin{definition}
    In the situation of Lemma \ref{lemma: central vertex}, we will refer to the $n$-valent vertex $V$ as the \textbf{central vertex} of $\mathsf{C}$. Also, we will call the connected components of $\mathsf{C}\smallsetminus \{V\}$ \textbf{leaves} of $\mathsf{C}$.  
\end{definition}

In light of the previous lemma, each leaf contains a unique marked point and all of its vertices are trivalent.

\begin{corollary}\label{cor: abstract multiplicity}
    The contribution of a curve $[h : \mathsf{C} \to \R^2]$ to
    $$
    \prod_{i=1}^n \mathrm{ev}_i^*(x_i) \cdot Z,
    $$
    where the points $p_i$ are in general position, is equal to the product of its local $\mathrm{ev}$-multiplicities at each vertex, as defined in \cite[Definition 3.18]{Goldner}.
\end{corollary}
\begin{proof}
    By the previous lemma the weights of $Z$ are all $1$, thus this immediately follows from \cite[Proposition 3.16]{Goldner} and \cite[Lemma 3.19]{Goldner}.
\end{proof}

\begin{corollary}\label{description curves}
    Let $[h : \mathsf{C} \to \R^2]$ be a curve that contributes to
    $$
    \prod_{i=1}^n \mathrm{ev}_i^*(x_i) \cdot Z,
    $$
    where the points are in general position. Then $\mathsf{C}$ has one of the two shapes in Proposition \ref{prop: type A and B}.
\end{corollary}
\begin{proof}
    Denote by $V$ the central vertex of $\mathsf{C}$ and let $\mathsf{L}$ be one of the leaves, which contains the marked point $p_i$. Then:
    \begin{itemize}
        \item[(1)] if $\mathsf{L}$ has no vertices, it has to be itself the marked point;
        \item[(2)] if $\mathsf{L}$ has one vertex, it must consist of one bounded edge, one end and the marked point $p_i$;
        \item[(3)] if $\mathsf{L}$ has two vertices, it must consist of two bounded edges, one marking and two ends. Moreover, the vertex with no marking attached has to be between $V$ and $p_i$. Otherwise we could find a string (i.e. an embedding $\R \to \mathsf{C}$ disjoint from any marked point), so we could deform the curve without changing its ev-multiplicity, and so it cannot contribute;
        \item[(4)]  $\mathsf{L}$ cannot have more than 2 vertices, or there would be again be a string in $\mathsf{L}$ and the contribution would be $0$.
    \end{itemize}

    \begin{figure}[h]
        \centering
        \begin{tikzpicture}[xscale=0.20,yscale=0.20]

            \draw[red, dashed] (0, 0) to (1, 5);
            \node at (0,0) {\small $\bullet$};
            \node at (-1.7,-1) {$V$};
            \node[red] at (2,6) {$p_i$};
            
            \draw[thick] (15, 0) to (13, 8);
            \draw[dashed, red] (14, 4) to (18, 5);
            \node at (15,0) {\small $\bullet$};
            \node at (13.3,-1) {$V$};
            \node[red] at (19,6) {$p_i$};

            \draw[thick] (30, 0) to (28, 4);
            \draw[thick] (28, 4) to (25, 7);
            \draw[thick] (28, 4) to (30, 8);
            \draw[dashed, red] (29, 6) to (32, 5);
            \node[red] at (33,4) {$p_i$};
            \node at (30,0) {\small $\bullet$};
            \node at (28.3,-1) {$V$};

            \draw[thick] (45, 0) to (44,3);
            \draw[thick] (44, 3) to (45, 6);
            \draw[thick] (44,3) to (41, 4);
            \draw[thick] (45, 6) to (44, 9);
            \draw[thick] (45,6) to (47, 8);
            \node at (45,0) {\small $\bullet$};
            \node at (43.3,-1) {$V$};

            \draw[thick, blue](44,3) to (43.67, 4);
            \draw[thick, blue] (45, 6) to (44, 5);
            \draw[thick, blue] (43.67, 4) to (44, 5);
            \draw[thick, blue] (43.67, 4) to (40.67, 5);
            \draw[thick, blue] (43, 8) to (44, 5);

            \draw[red, dashed] (47,7) to (50,6);
            \node[red] at (51,5) {$p_i$};
            
        \end{tikzpicture}
        \caption{The pictures for the situations (1), (2), (3) and (4) in the proof.}
        \label{fig:enter-label}
    \end{figure}
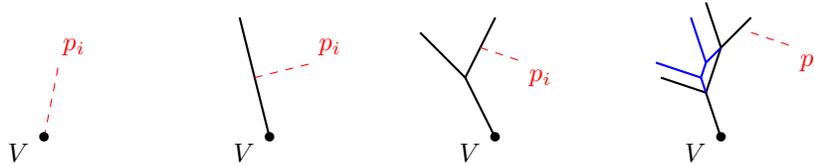
    If $L_i$ is the number of leaves with $i$ vertices, counting the number of markings we obtain $L_0 + L_1+L_2 = n$, whereas counting the number of ends, $L_1+2L_2= |\Delta| = 2(n-1)$. In particular, $2L_0+L_1=2$, which leaves only two possibilities, that correspond to the description of A and B.
\end{proof}

\subsection{Proof of Theorem \ref{thm: main}}




We now assume $X= \mathcal{H}_a$ and $\Sigma$ is the fan in \ref{notation: rays}. Corollary \ref{description curves} motivates the following definition.
\begin{definition}
    For a curve $[h: \mathsf{C} \to \R^2]$ contributing to \eqref{eqn: our count}, let $V$ be its central vertex and define: 
    \begin{align*}
        &\alpha([h])=\text{ number of leaves from } V \text{ having ends whose primitive vectors are } n_1 \text{ and } n_2,  \\
        &\beta([h])=\text{ number of leaves from } V \text{ having ends whose primitive vectors are } n_1 \text{ and } n_3,\\
        &\gamma([h])=\text{ number of leaves from } V \text{ having ends whose primitive vectors are } n_2 \text{ and } n_3,\\
        &\delta([h])=\text{ number of leaves from } V \text{ having ends whose primitive vectors are } n_3 \text{ and } n_4,\\
        &\chi([h])=\text{ number of leaves from } V \text{ having ends whose primitive vectors are } n_4 \text{ and } n_1,\\
        &\epsilon_1([h])= \text{ number of leaves from } V \text{ having one end whose primitive vector is } n_1,\\
        &\epsilon_2([h])= \text{ number of leaves from } V \text{ having one end whose primitive vector is } n_2,\\
        &\epsilon_3([h])= \text{ number of leaves from } V \text{ having one end whose primitive vector is } n_3,\\
        &\epsilon_4([h])= \text{ number of leaves from } V \text{ having one end whose primitive vector is } n_4.
    \end{align*}
\end{definition}

\begin{remark}\label{rmk: correspondence points-fork}
    Note that these are the possibilities that can occur. More precisely, let 
    \begin{align*}
    &\sigma_1= \{ x<0, ax+y>0\},\\
    &\sigma_2= \{ x>0, y>0\},\\
    &\sigma_3=\{ x>0, y<0\},\\
    &\sigma_4=\{ x<0, ax+y<0\}
    \end{align*}
    be the interiors of the maximal cones of the fan $\Sigma$ in Notation \ref{notation: rays}. Then leaves counted in $\alpha([h])$ corresponds to points $x_i$ in $\sigma_1$, leaves in $\beta([h])$ corresponds to points $x_i$ in $\sigma_1$ or $\sigma_2$, leaves in $\gamma([h])$ corresponds to points in $\sigma_2$, leaves in $\delta([h])$ corresponds to points in $\sigma_3$ and leaves in $\chi([h])$ corresponds to points in $\sigma_4$.
\end{remark}

\begin{lemma}
    For a curve $[h: \mathsf{C} \to \R^2]$ contributing to \eqref{eqn: our count}, the vector
    $$
    (\alpha,\beta,\gamma,\delta,\chi,\epsilon_1,\epsilon_2, \epsilon_3, \epsilon_4)=(\alpha([h]),\beta([h]),\gamma([h]), \delta([h]), \chi([h]),\epsilon_1([h]), \epsilon_2([h]), \epsilon_3([h]), \epsilon_4([h])).
    $$ 
    satisfies the following system
    \begin{equation}\label{eqn: our count bis}
    \left\lbrace \begin{array}{ll}
        \alpha + \beta +\chi &= |\mu_1|- \epsilon_1\\
        \alpha + \gamma &= |\mu_2|- \epsilon_2\\
        \beta + \gamma + \delta &= |\mu_3|- \epsilon_3\\
        \chi + \delta &= |\mu_4|- \epsilon_4
    \end{array}\right.
    \end{equation}
 \end{lemma}
 \begin{proof}
     The proof is clear.
 \end{proof}
In particular
$$
\alpha + \beta + \gamma + \delta + \chi =n-1- \delta_{B}^{\text{type of } [h]}
$$
where $\delta_{B}^{\text{type of } [h]}$ is $1$ when the type of $[h]$ is $B$ and $0$ when it is $A$. Solving the system in $\alpha$, we obtain
\begin{equation}\label{eqn: main system}
    \left\lbrace \begin{array}{ll}
        \beta &= n-1-\delta_{B}^{\text{type of } [h]}-|\mu_4|-|\mu_2|+ \epsilon_2 +\epsilon_4\\
        \gamma &= |\mu_2|- \epsilon_2 - \alpha\\
        \delta &= |\mu_3| + |\mu_4|- (n-1- \delta_{B}^{\text{type of } [h]})- \epsilon_3 -\epsilon_4+\alpha\\
        \chi &= n-1-\delta_{B}^{\text{type of } [h]} -|\mu_3| +\epsilon_3-\alpha
    \end{array}\right.
    \end{equation}
\begin{remark}
    If system \eqref{eqn: main system} does not have an admissible solution, then  $\tropTH =0$. In particular, if 
    $$
    |\mu_2|+|\mu_4|>n-1
    $$
    then $\beta<0$ (note that $\epsilon_2([h])$ and $\epsilon_4([h])$ cannot both be $1$ for $[h]$ contributing to \eqref{eqn: our count} being the $x_i$ in general position) and $\tropTH =0$. Similarly, if
    $$
    |\mu_3|>n-1
    $$
    then $\tropTH =0$. Finally, there is an isomorphism of $\Sigma$ preserving $n_2$ and $n_4$ and switching $n_1$ and $n_3$, so also
    $$
    |\mu_1|>n-1
    $$
    imlpies $\tropTH =0$
\end{remark} 

The next lemma computes the multiplicities of any curve to \ref{eqn: our count} for Hirzebruch surfaces.

\begin{lemma}\label{lem: contribution of curve}
    The contribution of a curve $[h : \mathsf{C} \to \R^2]$ to \ref{eqn: our count} is 
    $$
    \left(\prod_{i = 1}^4\prod_{j = 1}^{|\mu_{ij}|} \mu_{i,j}\right)a^{n-1-|\mu_2|-|\mu_4|}
    $$
\end{lemma}
\begin{proof}
    By Corollary \ref{cor: abstract multiplicity}, we need to compute the product of the $\mathrm{ev}$-multiplicity at each vertex.\\
    For a leaf $\mathsf{L}_i$ with two vertices, marking $x_i$, and weighted vectors associated to its end $\lambda_i \cdot n_{j_i}$, $\eta_i \cdot n_{k_i}$, the multiplicity at the vertex of $\mathsf{L}_i$ that is not adjacent to $x_i$  is precisely $\lambda_i \eta_i |\det(n_{j_i}, n_{k_i})|$. Note that for $j<k$,
        $$
        \det (n_j, n_k) = \left\lbrace 
        \begin{array}{ll}
            a & \text{ if }  j=1, k =3 \\
            0 & \text{ if }  j=2, k =4\\
            1 & \text{ otherwise}
        \end{array}
        \right. 
        $$
    \begin{itemize}
        \item If the curve is of type $A$, $V$ will have local multiplicity $1$ by \cite[definition 3.18]{Goldner}, and as $i$ goes from $2$ to $n$, we obtain a factor of $a$ for each leaf counted in $\beta ([h])= n-1-|\mu_2|-|\mu_4|$ (see Equation \ref{eqn: main system}), and the numbers $\lambda_i$ and $\eta_i$ go through all the elements of each partition, so we obtain the number in the statement of the lemma.
        \item If the curve is of type $B$, we can assume without lost of generality that the leaves with one end are $\mathsf{L}_1$, $\mathsf{L}_2$, with markings $x_1$, $x_2$ and that the end of $\mathsf{L}_i$ has weight $\lambda_i$ and primitive vector $n_{j_i}$, for $i = 1,2$.\\
        Then the leaves $\mathsf{L}_1$, $\mathsf{L}_2$ are precisely the fixed components of $V$, in the language of \cite[Definition 3.18]{Goldner}. We divide into cases:
        \begin{itemize}
            \item If $\epsilon_1([h]) + \epsilon_3([h])=1$, the $\mathrm{ev}$-multiplicity of $V$ is $\lambda_1\lambda_2$ and the local multiplicity for the rest of the vertices from the leaves $\mathsf{L}_3, \ldots , \mathsf{L}_n$ are counted as in the case $A$, so the multiplicity of the curve would be
            $$
            \left(\prod_{i = 1}^4\prod_{j = 1}^{|\mu_{ij}|} \mu_{i,j}\right)a^{\beta([h])}.
            $$
            Note that, by \ref{eqn: main system}, in this case $\beta([h]) = n-1-|\mu_4|-|\mu2|$ because $\epsilon_2([h]) + \epsilon_4([h]) = 1$.
            \item If $\epsilon_1([h]) + \epsilon_3([h])=2 $, the $\mathrm{ev}$-multiplicity of $V$ is  $a\lambda_1\lambda_2$ and the rest of local multiplicities are computed in the same way giving the number
            $$
            \left(\prod_{i = 1}^4\prod_{j = 1}^{|\mu_{ij}|} \mu_{i,j}\right)a^{\beta([h])+1}.
            $$
            In this case, by looking at \ref{eqn: main system}, $\beta([h])= n-2-|\mu_2| - |\mu_4|$.
        \end{itemize}
\end{itemize}
\end{proof}

Next, we explain which curves appear in \ref{eqn: our count}. We distinguish two cases.





\subsubsection{Case \texorpdfstring{$|\mu_3|+|\mu_4| \geq n-1$}{}}\label{sec: first case}

In this case, the following
\begin{equation}\label{eqn: sol1}
    \left\lbrace \begin{array}{ll}
        \bar{\alpha} &=0, \\
        \bar{\delta}_{B}^{\text{type of }h} &=0 \\
        \bar{\epsilon}_i &=0 \text{ for } i=1, \ldots, 4, \\
        \bar{\beta} &= n-1-|\mu_2|-|\mu_4|\\
        \bar{\gamma} &= |\mu_2|\\
        \bar{\delta} &= |\mu_3| + |\mu_4|- (n-1)\\
        \bar{\chi} &= n-1 -|\mu_3|
    \end{array}\right.
\end{equation}
is a solution of the system \eqref{eqn: main system}.

We put the points $x_1, \ldots , x_n$ in the plane $\R^2$ in general position and in such a way that: $x_1$ at $(0,0)$, and there are exactly $\bar{\alpha}$ points in $\sigma_1$, $\bar{\beta}+\bar{\gamma}$ in $\sigma_2$, $\bar{\delta}$ in $\{ ax+y>0, y<0\} \subseteq \sigma_3$ and $\bar{\chi}$ in $\{x<0, y<0 \} \subseteq \sigma_4$.

It is then clear that we can form 
$$
\prod_{i=1}^4 |\mu_i|!\binom{\bar{\beta}+\bar{\gamma}}{\bar{\beta}}
=
\prod_{i=1}^4 |\mu_i|!\binom{n-1-|\mu_4|}{|\mu_2|}
$$
curves $[h: \mathsf{C} \to \R^2]$ contributing to the intersection \eqref{eqn: our count} with domain $\mathsf{C}$ of type $A$ and $(\alpha ([h]), \ldots \chi([h]))=(\bar{\alpha}, \ldots ,\bar{\chi})$ as in Remark \ref{rmk: correspondence points-fork}. Here the factorial terms are counting the different ways of labelling the ends of $\mathsf{C}$. 
We will prove in \S\ref{sec:exclusion of further contributions} that these are all the contributing curves in this case.

\subsubsection{Case \texorpdfstring{$|\mu_3|+|\mu_4| < n-1$}{}}\label{sec: second case}

This case is more complicated. In this case, the following 
\begin{equation}\label{eqn: sol2}
    \left\lbrace \begin{array}{ll}
        \bar{\alpha} &=n-1-|\mu_3|-|\mu_4|, \\
        \bar{\delta}_{B}^{\text{type of }h} &=0 \\
        \bar{\epsilon}_i &=0 \text{ for } i=1, \ldots, 4, \\
        \bar{\beta} &= n-1-|\mu_4|-|\mu_2|\\
        \bar{\gamma} &= n-1-|\mu_1|\\
        \bar{\delta} &= 0\\
        \bar{\chi} &= |\mu_4|
    \end{array}\right.
\end{equation}
is a solution of \eqref{eqn: main system}.

As before, we put the points $x_1, \ldots , x_n$ in the plane $\R^2$ in general position and in such a way that: $x_1$ is at $(0,0)$, and there are exactly $\bar{\alpha}$ points in $\sigma_1$, $\bar{\beta}+\bar{\gamma}$ in $\sigma_2$, $\bar{\delta}$ in $\{ ax+y>0, y<0\} \subseteq \sigma_3$ and $\bar{\chi}$ in $\{x<0, y<0 \} \subseteq \sigma_4$.

We can form 
\begin{equation}\label{eqn: contribution 2.1}
\prod_{i=1}^4 |\mu_i|!\binom{\bar{\beta}+\bar{\gamma}}{\bar{\beta}}
=
\prod_{i=1}^4 |\mu_i|!\binom{|\mu_3|}{n-1-|\mu_4|-|\mu_2|}
\end{equation}
curves $[h: \mathsf{C} \to \R^2]$ contributing to the intersection \eqref{eqn: our count} with domain $\mathsf{C}$ of type $A$ as in the previous case. However, there also are some curves $[h: \mathsf{C} \to \R^2]$ of type $B$ which we now describe. The vertex $V$ is mapped to
$$
h(V) \in \{x >0, ax+y=0 \},
$$
and 
\begin{align*}
   (\epsilon_1 ([h]), \epsilon_2 ([h]), \epsilon_3([h]), \epsilon_4([h]), \alpha([h]), \beta([h]), \gamma([h]), \delta([h]), \chi([h]))=   (1, 1 , 0, 0, \bar{\alpha}-1,  \bar{\beta},  \bar{\gamma},  \bar{\delta},  \bar{\chi}).
\end{align*}
Note that the marking $x_1$ is reached from $V$ by a single bounded edge in the direction $n_1$.
There are 
\begin{align}\label{eqn: contribution 2.2}
\begin{split}
    &\prod_{i=1}^4 |\mu_i|!\sum_{k=1}^{\bar{\beta}} \binom{\bar{\alpha}+k-1}{\bar{\alpha}-1} \binom{\bar{\beta}+\bar{\gamma}-k}{\bar{\beta}-k}\\
    =&\prod_{i=1}^4 |\mu_i|! \sum_{k=1}^{n-1-|\mu_2|-|\mu_4|} \binom{n-2-|\mu_3|-|\mu_4|+k}{n-2-|\mu_3|-|\mu_4|} \binom{|\mu_3|-k}{n-1-|\mu_2|-|\mu_4|-k}
\end{split}
\end{align}
of such curves.

\begin{lemma}
    The curves in \eqref{eqn: contribution 2.1} and \eqref{eqn: contribution 2.2} are a total of 
    $$
    \prod_{i=1}^4 |\mu_i|!\binom{n-1-|\mu_4|}{|\mu_2|}
    $$
    curves.
\end{lemma}
\begin{proof}
    We will use the following two well-known combinatorial identities 
    \begin{equation}\label{eqn: combI}
    \binom{x}{y}=(-1)^y \binom{y-x-1}{y} \ 
    \end{equation}
    and
    \begin{equation}\label{eqn: combII}
    \binom{x}{y}=(-1)^{x-y} \binom{-y-1}{x-y} 
    \end{equation}
    valid for $x \in \Z_{>0}$ and $y \in \Z_{\geq 0}$, and Vandermonde identity
    \begin{equation}\label{eqn: Vandermonde}
    \sum_{k=0}^N \binom{x}{k} \binom{y}{N-k} =\binom{x+y}{N}
    \end{equation}
    valid for $ N \in \Z_{\geq 0}$ and $x,y \in \C$.

    We start with noticing that the $k=0$ term in the sum \eqref{eqn: contribution 2.2} is exactly the binomial coefficient in \eqref{eqn: contribution 2.1}. Therefore, the sum of the two contributions is (up to taking the product with $\prod_{i=1}^4 |\mu_i|!$)
    \begin{align*}
        &\sum_{k=0}^{n-1-|\mu_2|-|\mu_4|} \binom{n-2-|\mu_3|-|\mu_4|+k}{n-2-|\mu_3|-|\mu_4|} \binom{|\mu_3|-k}{n-1-|\mu_1|} \\
        =& (-1)^{n-1-|\mu_3|-|\mu_4|} \sum_{k=0}^{n-1-|\mu_2|-|\mu_4|} \binom{|\mu_3|+|\mu_4|-(n-1)}{k} \binom{|\mu_1|-n}{|\mu_1|+|\mu_3|-(n-1)-k} \\
        =& (-1)^{n-1-|\mu_3|-|\mu_4|} \sum_{k=-\infty }^{\infty} \binom{|\mu_3|+|\mu_4|-(n-1)}{k} \binom{|\mu_1|-n}{|\mu_1|+|\mu_3|-(n-1)-k} \\
        =& (-1)^{n-1-|\mu_3|-|\mu_4|} \binom{-|\mu_2|-1}{|\mu_1|+|\mu_3|-(n-1)} \\
        =& \binom{n-1-|\mu_4|}{n-1-|\mu_4|-|\mu_2|}
    \end{align*}
    where in the first equality we used 
    \eqref{eqn: combI}, in the second equality the fact that $|\mu_1|+|\mu_3|-(n-1)-k<0$ for $k > n-1-|\mu_2|-|\mu_4|$, in the third equality \eqref{eqn: Vandermonde} and in the last equality \eqref{eqn: combII}. This concludes the proof.
\end{proof}

We will prove in \S\ref{sec:exclusion of further contributions} that these are all the contributing curves in this case.

\subsubsection{Exclusion of further contributions}\label{sec:exclusion of further contributions}

We have to prove that the curves listed above are all the curves contributing to the intersection \ref{eqn: our count}. 

We start with adopting a unifying perspective.

\begin{remark}
    The solutions \eqref{eqn: sol1} and \eqref{eqn: sol2} are determined as the unique solution of the system \eqref{eqn: main system} where we set $\delta_{B}^{\text{type of } [h]}=0$, $\epsilon_i=0$ and ask for $\alpha$ to be the minimum possible.
\end{remark}

Call $(\bar{\alpha}, \bar{\beta},\bar{\gamma},\bar{\delta},\bar{\chi})$ such solution (so either as in \eqref{eqn: sol1} or as in \eqref{eqn: sol2})
\begin{remark}\label{rmk: easy sol -> hard sol}
    Suppose $(\alpha, \beta,\gamma,\delta,\chi)$ is a solution of \eqref{eqn: main system}, with $\delta_{B}^{\text{type of } [h]}=1$ and $\epsilon_i=1$ for exactly two indices $i$ and $\epsilon_i=0$ otherwise. Then $\alpha \geq \bar{\alpha}$ unless 
    $$
    (\epsilon_1, \epsilon_2, \epsilon_3, \epsilon_4, \alpha, \beta, \gamma, \delta, \chi)=(1,1,0,0, \bar{\alpha}+1, \bar{\beta}, \bar{\gamma}, \bar{\delta}, \bar{\chi}).
    $$
\end{remark}
\begin{proof}
    Suppose for example $\epsilon_1=\epsilon_2=0$ and $\epsilon_3=\epsilon_4=1$. Then $(\alpha, \beta, \gamma+1, \delta, \chi)$ is a solution of the system \eqref{eqn: main system} with $\delta_{B}^{\text{type of } [h]}=0$ and $\epsilon_i=0$ for all $i$. Therefore $\alpha \geq \bar{\alpha}$. The other cases are treated similarly. The case $\epsilon_2=\epsilon_4=1$ and $\epsilon_1=\epsilon_3=0$ is not possible because the points $x_i$ are in general positions and in particular there are no two of them in the same line $x=\mathrm{constant}$.
\end{proof}

We first exclude other contributions $[h: \mathsf{C} \to \R^2]$ with $\mathsf{C}$ of type $A$. For such $[h]$ we would have:
\begin{enumerate}
    \item[$\bullet$] if $h(V) \in \sigma_1$, then $\alpha([h]) \leq \bar{\alpha}$ and $\delta([h]) >\bar{\delta}$,
    \item[$\bullet$] if $h(V) \in \sigma_2$, then $\gamma([h])>\bar{\gamma}$,
    \item[$\bullet$] if $h(V) \in \sigma_3 \cap \{ ax+y>0\}$, then $\gamma([h])> \bar{\gamma}$, 
    \item[$\bullet$] if $h(V) \in \{ ax+y<0\}$, then $\alpha([h])+\beta([h])+\gamma([h]) >\bar{\alpha}+ \bar{\beta}+\bar{\gamma}$,
\end{enumerate}
each of which is in contradiction with  system \eqref{eqn: main system} and the choice of $\bar{\alpha}$. 

Assume now we are not in the exceptional situation of Remark \ref{rmk: easy sol -> hard sol}. Under this assumption, we now prove that there are no more contributions of type B. For such a $[h: \mathsf{C} \to \R^2]$ we would have one of the following contradictions:
\begin{enumerate}
    \item[$\bullet$] if $h(V) \in \{  y>0 \} \cup \{ x>0, y=0 \}$, then $\alpha([h])+ \beta([h])+ \gamma([h]) \leq \bar{\alpha}+\bar{\beta}+\bar{\gamma}-1$ and therefore by \eqref{eqn: main system} it must be $\epsilon_4([h])=0$, which would then imply $\alpha([h])+ \beta([h])+ \gamma([h]) = \bar{\alpha}+\bar{\beta}+\bar{\gamma}-2$ in contradiction with \eqref{eqn: main system},
    \item[$\bullet$] if $h(V) \in \{ x>0, ax+y>0 \}$, then $\chi([h]) > \bar{\chi}$ in contradiction with \eqref{eqn: main system} and our assumption on $\bar{\alpha}$,
    \item[$\bullet$] if $h(V) \in \{ x>0, ax+y=0 \}$, then $\epsilon_1([h])=1$ and $\epsilon_4([h])=0$ for the choice of the points $x_i$. If $\epsilon_3=0$, then system \eqref{eqn: main system} prescribes $\chi([h]) < \bar{\chi}$ but $\chi([h]) \geq \bar{\chi}$,
    \item[$\bullet$] if $h(V) \in \{ x \geq 0, ax+y <0 \}$ we have $\alpha([h])+\beta([h])+\gamma([h]) \geq \bar{\alpha}+ \bar{\beta}+ \bar{\gamma}$ and so, by \eqref{eqn: main system}, this is an equality and $\epsilon_4([h])=1$. However, $\epsilon_4([h])=0$ by the position of the points $x_i$,
    \item[$\bullet$] if $h(V) \in \{x<0, y<0\}$ then $\delta([h])+ \chi([h]) < \bar{\delta}+\bar{\chi}$ and so by \eqref{eqn: main system}, $\epsilon_4([h])=1$. It follows that $\epsilon_2([h])=0$ and so $\delta([h])+ \chi([h]) \leq \bar{\delta}+\bar{\chi}-2$, which is in contradiction with \eqref{eqn: main system},
    \item[$\bullet$] if $h(V) \in \{x<0, y=0\}$ then $\epsilon_3([h])=1$. If $\epsilon_4([h])=0$, then $\alpha([h])< \bar{\alpha}$ which is not possible for the choice of $\bar{\alpha}$ and our assumptions; if instead $\epsilon_4([h])=1$, then $\chi([h])< \bar{\chi}$ and so, by \eqref{eqn: main system}, $\alpha([h])>\bar{\alpha}$ which is in contradiction with the fact that there are no points $x_i$ which are reachable from $h(V)$ via leaves counted in $\delta([h])$ and are not reachable with the same type of leaves from $(0,0)$.
\end{enumerate}

Finally, we deal with curves of type B in the exceptional case of Remark \ref{rmk: easy sol -> hard sol}. In this case:
\begin{enumerate}
    \item[$\bullet$] $h(V)$ cannot be on $\{ ax+y \leq 0\ , x \geq 0 \} \smallsetminus \{(0,0)\}$ otherwise $\epsilon_1([h])=0$,
    \item[$\bullet$] $h(V) \in \{x <0, ax+y>0 \}$, then $\alpha([h]) \leq \bar{\alpha}-2$ which is not the case,
    \item[$\bullet$] $h(V)$ cannot be on $\{y>0 , x=0 \}$ otherwise $\epsilon_2([h])=0$,
    \item[$\bullet$] if $h(V) \in \{ x > 0 , ax+y >0 \} $, then $\chi([h])>\bar{\chi}$ which is also not the case,
    \item[$\bullet$] if $h(V) \in \{x>0, ax+y=0 \}$ then the curves of type B in \S\ref{sec: second case} appear,
    \item[$\bullet$] if $h(V) \in \{x>0, ax+y<0 \}$ then $\alpha([h])+\beta([h])+\gamma([h])\geq \bar{\alpha}+ \bar{\beta} +\bar{\gamma}$ which is not happening.
    
\end{enumerate}

This proves that the curves listed in \S\ref{sec: first case} and \S\ref{sec: second case} are all the curves contributing to the intersection \eqref{eqn: our count} and concludes the proof of Theorem \ref{thm: main}.

\section{Absence of rational curves interpolating \texorpdfstring{$n$}{} points on Hirzebruch surfaces}\label{sec: absence of curves}

Let $a \geq 2$ and assume that $\mu_{i,j}=1$ for all $i,j$. Let also $\beta \in H_2(X,\Z)$ be the unique curve class with $\beta . D_i=|\mu_i|$ for $i=1,2,3,4$. As in Remark \ref{rmk: absence of curves}, we have
$$
\mathsf{vTev}^{\mathcal{H}_a}_{\mathsf{\Gamma}}=0.
$$
By Lemma \ref{lemma: irreducibility of moduli spaces}, this is equivalent to say that the map \eqref{tau} is not dominant. We aim to now explain this phenomenon geometrically. 

By the proof of Theorem \ref{thm: correspondence}, it is enough to show that the restriction 
$$
\overset{\circ}{\tau}: \mathcal{M}_{\mathsf{\Gamma}}(\mathcal{H}_a) \to \mathcal{M}_{g,n} \times \mathcal{H}_a
$$ 
of \eqref{tau} is not dominant. By contradiction, suppose that $\overset{\circ}{\tau}$ is dominant. 

Let $\rho: \mathcal{H}_a=\P(\mathcal{O} \oplus \mathcal{O}(a)) \to \P^1$ be the projective bundle map and let
\begin{equation}\label{eqn: universal exact sequence}
0 \to \mathcal{O}_{\mathcal{H}_a}(-1) \to \rho^*( \mathcal{O}_{\P^1} \oplus \mathcal{O}_{\P^1}(a)) \to Q \to 0
\end{equation}
be the associated universal exact sequence.

\begin{lemma}\label{lemma: aux geometry}
    We have $2 |\mu_1| \geq n-1$.
\end{lemma}
\begin{proof}
    Composing maps in $\mathcal{M}_{\mathsf{\Gamma}}(\mathcal{H}_a)$ with $\rho$, we obtain maps $\P^1 \to \P^1$ of degree $|\mu_1|$. Let $\mathcal{M}_{\mathsf{\Gamma}'}(\P^1)$ be the moduli space of logarithmic stable maps $(\P^1,p_1,...,p_n,q_1,...,q_{2d}) \to \P^1$ with prescribed intersection multiplicities with $\partial \P^1$ all $1$. By Lemma \ref{lemma: irreducibility of moduli spaces}, this space is irreducible. It follows that the natural map $\mathcal{M}_{\mathsf{\Gamma}'}(\P^1) \to \mathcal{M}_{0,n} \times (\P^1)^n$ is also dominant and in particular that 
    $$
    2|\mu_1|+n-2=\mathrm{dim}(\mathcal{M}_{\mathsf{\Gamma}'}(\P^1)) \geq \mathrm{dim}(\mathcal{M}_{0,n} \times (\P^1)^n)=2n-3
    $$
    The conclusion follows.
\end{proof}
 
Given a map $f: \P^1 \to \mathcal{H}_a$, we can can pull back to $\P^1$ the exact sequence \eqref{eqn: universal exact sequence}, obtaining 
$$
\mathcal{O}_{\P^1} \oplus \mathcal{O}_{\P^1}(a|\mu_1|) \xrightarrow{v} \mathcal{O}_{\P^1}(|\mu_2| + a |   \mu_1|) \to 0
$$
such that 
$$
f(p)=\mathrm{ker}(v).
$$
We will think of $v$ as the data of two sections $s_1 \in H^0(\P^1,\mathcal{O}_{\P^1}(|\mu_2|))$ and $v_2 \in H^0(\P^1,\mathcal{O}_{\P^1}(|\mu_2|+a |\mu_1|)$.

The fact that $\overset{\circ}{\tau}$ is dominant amounts to say that for general points 
$$
\mathsf{L}_i \in \P( \mathcal{O}_{\P^1}(|\mu_2|) \oplus \mathcal{O}_{\P^1}(|\mu_2|+a |\mu_1|) ) \cong \mathcal{H}_a
\text{ and } 
p_1,\ldots,p_n \in \P^1
$$
there are sections $v_1 \in H^0(\P^1,\mathcal{O}_{\P^1}(|\mu_2|))$ and $v_2 \in H^0(\P^1,\mathcal{O}_{\P^1}(|\mu_2|+a |\mu_1|)$ such that 
$$
\langle v_1(p_i),v_2(p_i) \rangle =\mathsf{L}_i
$$
for each $i=1,\ldots,n$. Since the codimension of $$
\oplus_{i=1}^n \mathsf{L}_i \subseteq \oplus_{i=1}^n \mathcal{O}_{\P^1}(|\mu_2|)|_{p_i} \oplus \mathcal{O}_{\P^1}(|\mu_2|+a |\mu_1|)|_{p_i}
$$
is $n$,  the rank of the evaluation map 
$$
\Psi: H^0(\P^1,\mathcal{O}_{\P^1}(|\mu_2|)) \oplus H^0(\P^1,\mathcal{O}_{\P^1}(|\mu_2|+a |\mu_1|)) \to \bigoplus_{i=1}^n \mathcal{O}_{\P^1}(|\mu_2|)|_{p_i} \oplus \mathcal{O}_{\P^1}(|\mu_2|+a |\mu_1|)|_{p_i}
$$
must be at least $n+1$. Also, by Lemma \ref{lemma: aux geometry}, it is at most 
$$
2|\mu_2|+a |\mu_1|+2=2(n-1)-2|\mu_1| +2\leq n+1.
$$
Therefore, $2|\mu_1|=n-1$ and
$$
n-1=|\mu_2|+|\mu_4| \geq |\mu_4|=a |\mu_1|+|\mu_2| >n-1
$$
for $a\geq 2$. This yields a contradiction.

\section{Proof of Theorem \ref{thm: application}}\label{proof of applications}

It is well-known that for any $a \in \N$ the Hirzebruch surfaces $\mathcal{H}_a$ and $\mathcal{H}_{a+2}$ are deformation equivalent. We will use this fact to compute certain virtual Tevelev degrees of $\mathcal{H}_a$ by reducing to $\mathcal{H}_0 = \P^1 \times \P^1$ or $\mathcal{H}_1=\mathrm{Bl}_p \P^2$.  Theorem \ref{thm: application} will then follow by comparing the result with Theorem \ref{thm: main} above. 

Assume $ a\geq 2$.
\begin{lemma}\label{lemma: smooth family}
Let $0 < 2j \leq a$. Then there exists a smooth family $\pi: \mathcal{U} \to \mathbb{A}^1$ such that:
\begin{enumerate}
    \item $\pi^{-1}(0)$ is isomorphic to $ \mathcal{H}_a$;
    \item $\mathcal{U} \smallsetminus \pi^{-1}(0)$ is isomorphic to $\mathcal{H}_{(a -2j)} \times (\mathbb{A}^1 \smallsetminus \{0 \} )$ over $\mathbb{A}^1 \smallsetminus \{0 \}$;
    \item $\pi$ admits a section.
\end{enumerate}
\end{lemma}
\begin{proof}
This is probably a well-known construction. Let
$$
\pi: \mathcal{U}= \{ ([x_0,x_1],[y_0,y_1,y_2],s) \ | \ x_0^ay_1-x_1^ay_0+sx_0^jx_1^{j+1}y_2=0\} \to \mathbb{A}^1
$$
where $\mathcal{U} \subseteq \P^1 \times \P^2 \times \mathbb{A}^1$ and $\pi$ is the projection onto $\mathbb{A}^1$. We claim that $\pi$ is as stated in the lemma. In order to provide the required isomorphisms we will use the following description \cite{Cox, Manetti}:
$$
\mathcal{H}_a= (\C^2 \smallsetminus \{ 0 \} ) \times (\C^2 \smallsetminus \{ 0 \}) / \sim 
$$
where the equivalence relation $\sim$ is given by the $(\C^*)^2$-action is given by
$$
(\lambda,\eta) . (l_0,l_1,t_0,t_1)=(\lambda l_0, \lambda l_1, \lambda^a \eta t_0, \eta t_1).
$$
Then
\begin{align*} 
    \mathcal{H}_a &\to \pi^{-1}(0) \hspace{2.5cm} \\ 
    [l_0,l_1,t_0,t_1]& \mapsto ([l_0,l_1],[t_1l_0^a,t_1l_1^a,t_0])
\end{align*}
and 
\begin{align*} 
     \mathcal{H}_{(a -2m)} \times (\mathbb{A}^1 \smallsetminus \{0 \} ) &\to \mathcal{U} \smallsetminus \pi^{-1}(0) \\
     ([l_0,l_1,t_0,t_1],s) &\mapsto ([l_0,l_1],[sl_0^jt_0,sl_1^{j+1}t_1,l_1^{a-j-1}t_0-l_0^{a-j}t_1],s)
\end{align*}
are isomorphisms and 
\begin{align*}
s: \mathbb{A}^1 &\to \mathcal{U} \\
s \mapsto & ([0,1],[0,s,1])
\end{align*}
is a section.
\end{proof}

Write $a=2j$ or $a=2j+1$ for $j \in \Z_{\geq 1}$ depending on the parity of $a$ and let
$$
\pi: \mathcal{U} \to \mathbb{A}^1
$$
be a smooth family as in Lemma \ref{lemma: smooth family}. 

\begin{remark}
    Given a line bundle $L$ on $\mathcal{U} \smallsetminus \pi^{-1}(0)$, we can always extend $L$ to a line bundle on the all $\mathcal{U}$ (which is possible being $\mathcal{U}$ smooth). Even if there are many extensions of $L$, for each $s \in \mathbb{A}^1$ the restriction $L_s$ of $L$ to $\pi^{-1}(s)$ is independent of the extension.
\end{remark}

We treat the two case in Theorem \ref{thm: application} separately.

\subsubsection{Case \texorpdfstring{$a=2j$}{}}
   
Let $L$ be the pullback of $\mathcal{O}(1) \boxtimes \mathcal{O}(1)$ from $\mathcal{H}_0 \cong \P^1 \times \P^1$ on $\mathcal{U}\smallsetminus \pi^{-1}(0) \cong \mathcal{H}_0 \times (\mathbb{A}^1 \smallsetminus \{0 \})$.

\begin{lemma}\label{lemma: deformation class I}
    We have 
    $$
    c_1(L_0)=(1+j)D_1+D_2 \in A_*(\mathcal{H}_a)
    $$  
    where $D_1$ and $D_2$ are the toric divisors of $\mathcal{H}_a$ as in Notation \ref{notation: rays}
\end{lemma}

\begin{proof}
    Clearly 
    \begin{equation}\label{eq1 for L0}
    c_1(L_0)^2=c_1(L_1)^2=2.
    \end{equation}
    Also since the normal bundle to each fiber of $\pi$ is trivial, we have 
    \begin{equation}\label{eq2 for L_0}
    c_1(L_0).c_1(\mathcal{T}_{\pi^{-1}(0)})= c_1(L_1).c_1(\mathcal{T}_{\pi^{-1}(1)})=4
    \end{equation}
    These two equations completely determine $c_1(L_0)$. Namely, $D_1$ and $D_2$ form a basis of $A_1(\mathcal{H}_a)$ and writing $c_1(L_1)= x D_1 +y D_2$, Equations \ref{eq1 for L0} and \ref{eq2 for L_0} yields
    $$
    \begin{cases}
        2=2xy-y^2a,\\
        4= 2x+(2+a)y-2ya  \, 
    \end{cases} 
    $$
    whose integral solution is $(x,y)=(1+j,1)$.
\end{proof}
        
Fix $d>0$ such that $2d=n-1$ and let $\beta =d[(1+j)D_1+D_2]$. Since Gromov-Witten invariants are deformation invariant, we get
$$
1=\mathsf{vTev}^{\P^1 \times \P^1}_{0,n,(d,d)}=\mathsf{vTev}^{\mathcal{H}_a}_{0,n,\beta}
$$
The first equality follows from \cite[Example 2.2 and Proposition 2.3]{BP}. Note that the existence of a section of $\pi$ guarantee that on each fiber the point class can be realized as the restriction of a class from $\mathcal{U}$.

To conclude, if $\mathsf{\Gamma}$ and $\alpha$ are as in Theorem \ref{thm: application}, we have
$$
\mathsf{vTev}_\mathsf{\Gamma}^{\mathcal{H}_a}=0 \neq 1 = \mathsf{vTev}^{\mathcal{H}_a}_{0,n,\beta}
$$
and thus $\alpha_*[\oM_\mathsf{\Gamma}(X)]^{\mathrm{vir}} \neq [\oM_{0,n}(X,\beta)]^{\mathrm{vir}}$ as desired.

\subsubsection{Case \texorpdfstring{$a=2j+1$}{}}

This case is very similar, but uses \cite[Theorem 14]{CL} instead of \cite{BP}.

Let $L_\H$ (resp. $L_\E$) be the pullback of $\mathcal{O}(\H)$ (resp $\mathcal{O}(\E)$) from $\mathcal{H}_1 \cong \mathrm{Bl}_p \P^2$ on $\mathcal{U}\smallsetminus \pi^{-1}(0) \cong \mathcal{H}_1 \times (\mathbb{A}^1 \smallsetminus \{0 \})$. Here $\H$ (resp. $\E$) is the hyperplane class (resp. the class of the exceptional divisor) on $\mathrm{Bl}_p \P^2$.

\begin{lemma}
    We have 
    $$
    c_1((L_\H)_0)=(1+j)D_1+D_2
    $$
    and 
    $$
    c_1((L_\E)_0)=jD_1+D_2
    $$
\end{lemma}

\begin{proof} 
    Proceed as in Lemma \ref{lemma: deformation class I}
\end{proof}
        
Fix $d$ and $k$ integers such that $0 \leq k \leq d$, $0 \leq k \leq n-1-d$ and $3d-k=2(n-1)$. Call $\beta =[j(d-k)+d] D_1 +(d-k) D_2$. Then
$$
\binom{n-1-d}{k}=\mathsf{vTev}^{\mathrm{Bl}_p \P^2}_{0,n,d \H-k \E}=\mathsf{vTev}^{\mathcal{H}_a}_{0,n,\beta}
$$
where the first equality follows from \cite[Theorem 14]{CL}. If $\mathsf{\Gamma}$ and $\alpha$ are as in Theorem \ref{thm: application}, we have
$$
\mathsf{vTev}_\mathsf{\Gamma}^{\mathcal{H}_a}=0 \neq \binom{n-1-d}{k} = \mathsf{vTev}^{\mathcal{H}_a}_{0,n,\beta}
$$
and thus $\alpha_*[\oM_\mathsf{\Gamma}(X)]^{\mathrm{vir}} \neq [\oM_{0,n}(X,\beta)]^{\mathrm{vir}}$.

This concludes the proof of Theorem \ref{thm: application}.


\newcommand{\etalchar}[1]{$^{#1}$}

\end{document}